\documentclass[a4paper,10pt]{amsart}
\usepackage{amsmath,amsthm,xypic,
amssymb,
pdfsync
}
\newtheorem{thm}{Theorem}[section]

\newtheorem{prop}[thm]{Proposition}
\newtheorem{cor}[thm]{Corollary}

\newtheorem{lem}[thm]{Lemma}
\theoremstyle{definition}

\newtheorem{rem}[thm]{Remark}

\theoremstyle{plain}

\renewcommand{\phi}{\varphi}
\newcommand{\fix}{\mathop{\mathrm{Fix}}}

\newcommand{\sym}{\mathop{\mathrm{Sym}}}

\newcommand{\alt}{\mathop{\mathrm{Alt}}}
\newcommand{\nrd}{\mathop{\mathrm{Nrd}}}

\newcommand{\End}{\mathop{\mathrm{End}}}

\newcommand{\car}{\mathop{\mathrm{char}}}
\newcommand{\id}{\mathop{\mathrm{id}}}

\newcommand{\disc}{\mathop{\mathrm{disc}}}

\newcommand{\ad}{\mathop{\mathrm{Ad}}}

\newcommand{\Sym}{{\rm Sym}}

\title[{On split products of quaternion algebras with involution}]{On split products of quaternion algebras with involution in characteristic two}
\author{M. G. Mahmoudi, \ A.-H. Nokhodkar}
\thanks{\hspace{-.2cm}{\scriptsize{\em  E-mail Addresses}: {\tt
    mmahmoudi@sharif.ir} (M. G. Mahmoudi),
   {\tt
    anokhodkar@yahoo.com} (A.-H. Nokhodkar).}}
\begin{document}
\maketitle
\vspace{-.6cm}
\begin{center}
{\footnotesize
{\em Department of Mathematical Sciences, Sharif University of Technology,\\
 P. O. Box 11155-9415, Tehran, Iran.}}
\end{center}
\begin{abstract}
The question of whether a split tensor product of quaternion algebras with
involution over a field of characteristic two
 can be expressed as a tensor product of split quaternion algebras with
 involution,
is shown to have an affirmative answer.\\
\end{abstract}

\begin{center}
\end{center}

\section{Introduction}
In \cite[Ch. 9]{shapiro}, D. B. Shapiro formulates the following conjecture:\\

In the category of $F$-algebras with involution of the first kind, suppose $(A,\sigma)\simeq
(Q_1,\sigma_1)\otimes\cdots\otimes(Q_n,\sigma_n)$ where
each $(Q_k,\sigma_k)$ is a quaternion algebra with involution.
If the algebra $A$ is split, then there is a decomposition
$$(A,\sigma)\simeq
(Q'_1,\sigma'_1)\otimes\cdots\otimes(Q'_n,\sigma'_n)$$
where each $(Q'_k,\sigma'_k)$ is a split quaternion algebra with involution.\\

For the case where the characteristic of $F$ is different from $2$,
many authors have studied this conjecture either in its current
formulation or in its equivalent variations, including the Pfister
factor conjecture in the theory of
spaces of similarities (developed by Shapiro), Hurwitz problem of
composition of quadratic forms and Pfister involutions
(see \cite{wads-shap}, \cite{yuzvinsky}, \cite{shapiro},
\cite{fluckiger}, \cite{serhir}, \cite{sivatski}, \cite{grenier},
\cite{garibaldi}, \cite{karpenko}).
Finally in \cite{becher}, Becher proved this conjecture in the case of
characteristic different from $2$.

Although, Shapiro formulates this conjecture only for the case of
characteristic different from $2$, but he asks if there is some sort of
Pfister factor conjecture in characteristic $2$ (see
\cite[p. 371]{shapiro}) and cites a work by Baeza's student Junker in
this direction.
See also \cite{elduque}.
Thus it would be interesting to
investigate the above conjecture in the case of characteristic $2$.
This is the purpose of our work.

\section{Preliminaries}
Let $F$ be a field of characteristic $2$ and let $V$ be an $n$-dimensional vector space over $F$.
A bilinear form $b:V\times V\rightarrow F$ is called {\it alternating} if $b(v,v)=0$ for every $v\in V$.
Otherwise $b$ is called {\it nonalternating}.
It is known that a symmetric bilinear form $b$ is nonalternating if
and only if it is diagonalizable (see \cite[(1.17)]{elman}), i.e., there exists a basis $\{v_1,\cdots,v_n\}$ of $V$ such that $b(v_i,v_j)=0$ for every $i\neq j$.
Such a basis is called an {\it orthogonal basis} of $(V,b)$.
In this case if $b(v_i,v_i)=c_i$, $1\leqslant i\leqslant n$, the
bilinear form $b$ is denoted by $\langle c_1,\cdots,c_n\rangle$.

A {\it quadratic form} over $V$ is a map $q:V\rightarrow F$ such that $q(\alpha u+\beta v)=\alpha^2q(u)+\beta^2q(v)+\alpha\beta b(u,v)$ for every $\alpha,\beta\in F$ and $u,v\in V$,
where $b:V\times V\rightarrow F$ is a symmetric bilinear form.
The quadratic form $q$ is called {\it regular} if the bilinear form $b$ is nondegenerate.
By a {\it quadratic space} $(V,q)$ we mean a vector space
$V$ over $F$ together with a regular quadratic form $q$ over $V$.
Every quadratic space $(V,q)$ has a {\it symplectic basis}, i.e., a basis $\mathcal{B}=\{u_1,v_1,\cdots,u_n,v_n\}$ such that $b(u_i,v_i)=1$, $i=1,\cdots,n$, and $b(x,y)=0$ for every other pairs of vectors $x,y\in\mathcal{B}$.
A two dimensional quadratic space is called a {\it quadratic plane} and is denoted by $(\mathbb{E},\phi)$.

The group of all isometries of a quadratic space $(V,q)$, i.e., the {\it orthogonal group}
of $(V,q)$, is denoted by $O(V,q)$.
For an isometry $\tau\in O(V,q)$ we use the notation $\fix(V,\tau)=\{v\in V|\tau(v)=v\}$.
An isometry $\tau\in O(V,q)$ is called an {\it involution} if $\tau^2=\id$.
If $u\in V$ is an anisotropic vector, the involution $\tau_u\in O(V,q)$ defined by $\tau_u(v)=v+\frac{b(v,u)}{q(u)}u$ for every $v\in V$, is called the {\it (orthogonal) reflection} along $u$.
For an isometry $\tau\in O(V,q)$, the {\it spinor norm} of $\tau$ is denoted by $\theta(\tau)$.

Let $(\mathbb{A},q)$ be a quadratic space of dimension $4$ over a field $F$ of characteristic $2$.
An involution $\tau\in O(\mathbb{A},q)$ is called an {\it interchange isometry} if $\fix(\mathbb{A},\tau)$ is a two dimensional totally isotropic space.
If $(V,q)$ is a quadratic space over $F$, then for every involution
$\tau$ in $O(V,q)$, there exists an orthogonal decomposition
\begin{eqnarray}
V=W\perp\mathbb{E}_1\perp\cdots\perp\mathbb{E}_r\perp\mathbb{A}_1\perp\cdots\perp\mathbb{A}_s,
\end{eqnarray}
with the following properties:

(i) $\tau|_W=\id_W$;

(ii) each $\mathbb{E}_i$ is a quadratic subplane of $V$ and $\tau|_{\mathbb{E}_i}$ is a reflection;

(iii) each $\mathbb{A}_i$ is a $4$-dimensional subspace of $V$ and $\tau|_{\mathbb{A}_i}$ is an interchange isometry (see  \cite[Thm. 1]{wiitala}).

The decomposition $(1)$ is called a {\it Wiitala decomposition} of $(V,\tau)$ and the subspace $W$ is called a {\it maximal fixed orthogonal summand} of $V$.
Also for every involution $\id\neq\tau\in O(V,q)$, there exists a Wiitala decomposition of $(V,\tau)$ with exactly one of the following additional properties:

(a) There is no interchange isometry in this decomposition, i.e., $s=0$.

(b) There is no reflection in this decomposition, i.e., $r=0$ (see  \cite[Thm. 2]{wiitala}).

We say that $\tau$ is of {\it reflectional} (resp. {\it  interchanging}) kind if $V$ has a decomposition of the form (a) (resp. (b)).

Throughout this paper when we say $A$ is a {\it central simple $F$-algebra},
we implicitly suppose that $F$ is the center of $A$.
For an {\it involution} $\sigma$ (i.e., an anti-automorphism with
$\sigma^2=\id$) on a central simple $F$-algebra $A$, the set of {\it
  symmetric} and {\it alternating} elements of $A$ are defined as
follows:
$$\Sym(A,\sigma)=\{a\in A | \sigma(a)=a\},\
\alt(A,\sigma)=\{a-\sigma(a) | a\in A\}.$$
If $\car F=2$, an involution $\sigma$ on $A$ is of symplectic type if and only if $1\in\alt(A,\sigma)$ (see \cite[(2.6)]{knus}).
If $\sigma$ is of orthogonal type and $A$ is of even degree $n=2m$ over $F$, then the {\it discriminant} of $\sigma$ is defined as follows:
$$\disc\sigma=(-1)^m{\nrd}_A(a)F^{\times2}\in F^\times/F^{\times2} \hspace{.3 cm} \textrm{for}~ a\in \alt(A,\sigma)\cap A^*,$$
where $\nrd_A(a)$ is the reduced norm of $a$ and $A^*$ is the unit group of $A$.

The Clifford algebra of a quadratic space $(V,q)$ is denoted by $C(V)$.
It is easy to see that every involution $\tau\in O(V,q)$, induces an involution $J_\tau$ of the first kind on $C(V)$ such that $J_\tau(v)=\tau(v)$ for every $v\in V$.

If $b:V\times V\rightarrow F$ is a nondegenerate bilinear form,
the {\it adjoint involution} of $\End_F(V)$ with respect to $b$ is the unique involution $\ad_b$ on $\End_F(V)$ characterized by the property
$b(x,f(y))=b(\ad_b(f)(x),y)$ for every $x,y\in V$ and $f\in \End_F(V)$.

The transpose involution on $M_n(F)$, i.e., the algebra of all
$n\times n$ matrices over a field $F$, is denoted by $t$.
The {\it canonical involution} $\gamma$ on $M_2(F)$ is defined by
{\setlength\arraycolsep{2pt}
\begin{eqnarray*}
\gamma\left(\begin{array}{cc}a & b \\c & d\end{array}\right)=\left(\begin{array}{cc}d & -b \\-c & a\end{array}\right),
\end{eqnarray*}}
for $a,b,c,d\in F$.
It is known that the involution $\gamma$ is the unique symplectic
involution on $M_2(F)$ and it is characterized by the property
$\gamma(x)x\in F$ for every $x\in M_2(F)$
(see \cite[Ch. I]{knus}).

We need the following results:
\begin{prop} \label{clif}{\rm(\cite[(6.1) and (4.11)]{mahmoudi})}
Let $(Q,\sigma)$ be a quaternion algebra with an involution of the first kind over a field $F$ of characteristic $2$.
Then there is a quadratic plane $(\mathbb{E},\phi)$ over $F$ and an involution $\tau$ in $O(\mathbb{E},\phi)$ such that $(Q,\sigma)\simeq (C(\mathbb{E}),J_\tau)$.
Furthermore if $\sigma$ is of symplectic type then $\tau$ is necessarily equal to identity and if $\sigma$ is of orthogonal type then $\tau$ is necessarily a reflection.
In the latter case we have $\disc J_\tau=\theta(\tau)$.
\end{prop}

\begin{prop}\label{cint}{\rm(\cite[(6.10)]{mahmoudi})}
Let $(\mathbb{A},q)$ be a $4$-dimensional quadratic space over a field $F$ of characteristic $2$ and let $\tau$ be an interchange isometry of $(\mathbb{A},q)$.
Then $(C(\mathbb{A}),J_\tau)\simeq(C(\mathbb{E}_1),J_{\tau_1})\otimes(C(\mathbb{E}_2),J_{\tau_2})$ where
$(\mathbb{E}_i,\phi_i)$ is a suitable quadratic plane over $F$, $\tau_i$ is a reflection of $(\mathbb{E}_i,\phi_i)$ and $\theta(\tau_i)=1\in F^\times/F^{\times2}$, $i=1,2$.
\end{prop}

\begin{prop}\label{orthog}{\rm(\cite[(4.7)]{mahmoudi})}
Let $(V,q)$ be a quadratic space over a field $F$ of characteristic $2$ and let $\tau$ be an involution in $O(V,q)$.
Then the involution $J_\tau$ on $C(V)$ is of orthogonal type if and
only if $(V,\tau)$ has trivial maximal fixed orthogonal summand if and
only if $\dim \fix(V,\tau)=\frac{1}{2}\dim V$.
\end{prop}

\section{Involutions on split Clifford algebra in low dimensions}\label{sec-clif}
Let $F$ be a field of characteristic $2$ and let $\alpha\in F^\times$.
Consider the involution $T_\alpha:M_2(F)\rightarrow M_2(F)$ defined by
{\setlength\arraycolsep{2pt}
\begin{eqnarray*}
T_\alpha\left(\begin{array}{cc}a & b \\c & d\end{array}\right)=\left(\begin{array}{cc}a & c\alpha^{-1} \\b\alpha & d\end{array}\right).
\end{eqnarray*}}
In particular $T_1=t$.
Note that $T_\alpha$ is, up to isomorphism, the unique involution of orthogonal type on $M_2(F)$ such that  $\disc T_\alpha=\alpha F^{\times2}\in F^\times/F^{\times2}$ (\cite[(7.4)]{knus}).

\begin{rem}\label{sp}
Let $(\mathbb{E},\phi)$ be a quadratic plane over a field $F$ of characteristic $2$.
Then $C(\mathbb{E})$ splits if and only if $\phi$ represents $1$, see \cite[(9.6 (2))]{elman}, \cite[(11.2 (4))]{elman} and \cite[(12.5)]{elman}.
It follows that $(C(\mathbb{E}),J_\tau)\simeq(M_2(F),t)$ if and only if $\tau$ is a reflection and $\theta(\tau)=1\in F^\times/F^{\times2}$.
More generally, if $C(\mathbb{E})$ splits then
$(C(\mathbb{E}),J_\tau)\simeq(M_2(F),T_\alpha)$ if and only if $\tau$
is a reflection and $\theta(\tau)=\alpha F^{\times2}$, also
$(C(\mathbb{E}),J_{\id})\simeq(M_2(F),\gamma)$.
\end{rem}

The proof of the following simple observation is left to the reader.

\begin{lem}\label{totimes}
Let $F$ be a field of characteristic $2$ and let $A\in M_n(F)$ such that $A^t=A$ and $A^2\in F$. Then $A^2\in F^2$.
\end{lem}

\begin{prop}\label{trans}
Let $(V,q)$ be a quadratic space over a field $F$ of characteristic
$2$ and let $\tau\in O(V,q)$ be an involution.
Then $(C(V),J_\tau)\simeq(M_{2^n}(F),t)$ if and only if $\dim\fix(V,\tau)=\frac{1}{2}\dim V$ and $q(x)\in F^2$ for every $x\in\fix(V,\tau)$.
\end{prop}

\begin{proof}
Since the involution $t$ is of orthogonal type, if $f:(C(V),J_\tau)\simeq(M_{2^n}(F),t)$ is an isomorphism, then by (\ref{orthog}), we have $\dim\fix(V,\tau)=\frac{1}{2}\dim V$.
Let $x\in \fix(V,\tau)$, i.e., $\tau(x)=x$ and set $A=f(x)\in M_{2^n}(F)$.
Then $A^2=f(x)^2=q(x)\in F$ and $A^t=A$, so by (\ref{totimes}), $A^2\in F^2$, i.e., $q(x)=x^2\in F^2$.

Conversely suppose that $\dim\fix(V,\tau)=\frac{1}{2}\dim V$ and $q(x)\in F^2$ for every $x\in\fix(V,\tau)$.
Then by (\ref{orthog}), $(V,\tau)$ has trivial maximal fixed orthogonal summand, so $(V,\tau)$ has a Wiitala decomposition as $V=\mathbb{E}_1\perp\cdots\perp\mathbb{E}_r\perp \mathbb{A}_1\perp\cdots\perp \mathbb{A}_s$.
By (\ref{cint}) and (\ref{sp}), we have $(C(\mathbb{A}_i),J_{\tau|_{\mathbb{A}_i}})\simeq(M_4(F),t)$, $i=1,\cdots,s$.
Also since $q(x)\in F^2$ for every $x\in\fix(V,\tau)$, we obtain $\theta(\tau|_{\mathbb{E}_i})=1\in F^{\times}/F^{\times2}$, $i=1,\cdots,r$.
So by (\ref{sp}) we have $(C(\mathbb{E}_i),J_{\tau|_{\mathbb{E}_i}})\simeq(M_2(F),t)$, $i=1,\cdots,r$.
This completes the proof.
\end{proof}

\begin{rem}\label{int4}
Let $(\mathbb{A},q)$ be a four dimensional quadratic space over a field $F$ of characteristic $2$ and let $\tau$ be an involution in $O(\mathbb{A},q)$.
Then by (\ref{orthog}),  we have $\dim\fix(\mathbb{A},\tau)=2$ if and only if either $\tau$ is an interchange isometry or $\tau=\tau_1\perp\tau_2$ is an orthogonal sum of two reflections.
\end{rem}

\begin{lem}\label{theta}
For $i=1,\cdots,n$, let $(\mathbb{E}_i,\phi_i)$ be a quadratic plane over a field $F$ of characteristic $2$ and let $\tau_i$ be a reflection of $(\mathbb{E}_i,\phi_i)$.
Set $(V,q)=(\mathbb{E}_1,\phi_1)\perp\cdots\perp(\mathbb{E}_n,\phi_n)$ and $\tau=\tau_1\perp\cdots\perp\tau_n$.
Then we have $\theta(\tau_i)=1\in F^\times/F^{\times2}$ for every $i=1,\cdots,n$, if and only if $q(x)\in F^2$ for every $x\in\fix(V,\tau)$.
\end{lem}

\begin{proof}
Choose an anisotropic vector $u_i\in\mathbb{E}_i$, $i=1,\cdots,n$, such that $\tau_i=\tau_{u_i}$.
Then $\{u_1,\cdots,u_n\}$ is an orthogonal basis of $\fix(V,q)$.
We have $q(x)\in F^2$ for every $x\in\fix(V,\tau)$ if and only if $q(u_i)\in F^2$ for every $i=1,\cdots,n$.
As $\theta(\tau_i)=q(u_i)F^{\times2}\in F^\times/F^{\times2}$, this is equivalent to the condition that $\theta(\tau_i)=1\in F^\times/F^{\times2}$ for every $i=1,\cdots,n$.
\end{proof}

\begin{cor}\label{tint}
Let $(\mathbb{A},q)$ be a $4$-dimensional quadratic space over a field $F$ of characteristic $2$ and let $\tau$ be an involution in $O(\mathbb{A},q)$.
Then $(C(\mathbb{A}),J_\tau)\simeq(M_4(F),t)$ if and only if either $\tau$ is an interchange isometry or $\tau=\tau_1\perp\tau_2$ is an orthogonal sum of two reflections $\tau_1$ and $\tau_2$ with $\theta(\tau_1)=\theta(\tau_2)=1\in F^\times/F^{\times2}$.
\end{cor}

\begin{proof}
If $(C(\mathbb{A}),J_\tau)\simeq(M_4(F),t)$, then by (\ref{trans}) we have $\dim\fix(\mathbb{A},\tau)=2$ and $q(x)\in F^2$ for every $x\in\fix(\mathbb{A},\tau)$.
So by (\ref{int4}) either $\tau$ is an interchange isometry or $\tau=\tau_1\perp\tau_2$ is an orthogonal sum of two reflections.
In the latter case by (\ref{theta}) we have $\theta(\tau_1)=\theta(\tau_2)=1\in F^\times/F^{\times2}$.

Conversely if $\tau$ is an interchange isometry then $\dim\fix(\mathbb{A},\tau)=2$ and for every $x\in\fix(\mathbb{A},\tau)$ we have $q(x)=0\in F^2$, so (\ref{trans}) implies that $(C(\mathbb{A}),J_{\tau})\simeq(M_4(F),t)$.
Finally, if $\tau=\tau_1\perp\tau_2$ is an orthogonal sum of two reflections, then $\dim\fix(\mathbb{A},\tau)=2$.
Also since $\theta(\tau_1)=\theta(\tau_2)=1\in F^\times/F^{\times2}$,  by (\ref{theta}) we have $q(x)\in F^2$ for every $x\in\fix(\mathbb{A},\tau)$.
Therefore (\ref{trans}) implies that $(C(\mathbb{A}),J_{\tau})\simeq(M_4(F),t)$.
\end{proof}

\section{Split products of quaternion algebras with involution}\label{sec-pfister}

\begin{rem}\label{alt}
Let $(\mathbb{E},\phi)$ be a quadratic plane over a field $F$ of characteristic $2$ and let $u\in\mathbb{E}$ be an anisotropic vector.
Extend $\{u\}$ to a symplectic basis $\{u,v\}$ of $\mathbb{E}$ and set $\tau=\tau_u$.
Then $J_\tau(v)=v+\frac{1}{\phi(u)}u$.
So $\alt(C(\mathbb{E}),J_\tau)=Fu$.
We also have $u^2=q(u)\in F^\times$.
In particular by (\ref{clif}), if $(Q,\sigma)$ is a quaternion algebra with involution of ortho\-gonal type over $F$ and $r\in\alt(Q,\sigma)$, then $r^2\in F^\times$.
Note that we also have $\disc\sigma=\alpha F^{\times2}\in F^\times/F^{\times2}$ where $\alpha=r^2\in F^\times$.
\end{rem}

\begin{lem}\label{b}
Let $K/F$ be an extension of fields of characteristic $2$ and let $\sigma$ be an involution of orthogonal type on $M_2(K)$ with $\disc\sigma=\alpha K^{\times2}$, $\alpha\in F^\times$.
Then there is a nonalternating symmetric bilinear form $(W,b)$ over $K$ with an orthogonal basis $\{u,v\}$ such that $(M_2(K),\sigma)\simeq(\End_K(W),\ad_b)$, $b(u,u)=1\in F^\times$ and $b(v,v)=\alpha\in F^\times$.
\end{lem}

\begin{proof}
Let $(W,b)$ be the nonalternating bilinear form $\langle1,\alpha\rangle$ over $K$ with an orthogonal basis $\{u,v\}$.
Then $(\End_K(W),\ad_b)$ is an algebra with involution of orthogonal type.
Since the involution $\sigma$ is of orthogonal type and $\disc\sigma=\alpha K^{\times2}=\disc b=\disc \ad_b$, by (\cite[(7.4)]{knus}) we have $(M_2(K),\sigma)\simeq(\End_K(W),\ad_b)$.
Moreover we have $b(u,u)=1$ and $b(v,v)=\alpha$.
\end{proof}

\begin{rem}
Let $K/F$ be a field extension and let $W$ be a vector space over $K$.
Let $\{w_1,\cdots,w_n\}$ be a basis of $W$ and set $V=Fw_1+\cdots+Fw_n$.
Using the $F$-algebra monomorphism $\phi:\End_F(V)\hookrightarrow\End_{K\otimes F}(K\otimes V)$ defined by $\phi(f)=\id_K\otimes f$, we may consider $\End_F(V)$ as an $F$-subalgebra of $\End_K(W)$.
\end{rem}

The following result is implicitly contained in the proof of \cite[(3.4 (2))]{bayer}:
\begin{lem}\label{ad}
Let $K/F$ be a field extension and let $(W,b)$ be a nondegenerate
nonalternating symmetric bilinear space over $K$.
Let $\{w_1,\cdots,w_n\}$ be an orthogonal basis of $W$ and $V=Fw_1+\cdots+Fw_n$.
If $b(w_i,w_i)\in F$ for every $i=1,\cdots,n$, then $\ad_b(\End_F(V))=\End_F(V)$.
\end{lem}

\begin{proof}
Let $f\in \End_F(V)$. We want to show that $\ad_b(f)\in \End_F(V)$.
It is enough to show that $\ad_b(f)(w_i)\in V$ for every $i=1,\cdots,n$.
Write $f(w_i)=\alpha_{1i}w_1+\cdots+\alpha_{ni}w_n$ where $\alpha_{1i},\cdots,\alpha_{ni}\in F$ and
write $\ad_b(f)(w_i)=\beta_{1i}w_1+\cdots+\beta_{ni}w_n$ where $\beta_{1i},\cdots,\beta_{ni}\in K$.
We claim that for every $i, j=1,\cdots,n$, $\beta_{ji}\in F$.
Since $b(w_i,f(w_j))=b(\ad_b(f)(w_i),w_j)$, we obtain
\begin{eqnarray}\label{1}
\alpha_{ij}b(w_i,w_i)=\beta_{ji}b(w_j,w_j).
\end{eqnarray}
The left side of (\ref{1}) belongs to $F$.
Also since $b$ is nondegenerate and $\{w_1,\cdots,w_n\}$ is an orthogonal basis of $W$, we have $b(w_k,w_k)\neq0$ for every $k=1,\cdots,n$.
So $\beta_{ji}\in F$ for every $i, j=1,\cdots,n$.
This completes the proof.
\end{proof}

\begin{lem}\label{discQ}
Let $(A,\sigma)$ be a central simple $F$-algebra with involution of
the first kind where $F$ is a field of arbitrary characteristic.
Let $B$ be a central simple $F$-subalgebra of $A$ such that $\sigma(B)=B$ and let $C=C_A(B)$ be the centralizer of $B$ in $A$.
Then $\alt(A,\sigma)\cap B=\alt(B,\sigma|_B)$ except for the case where $\car F=2$ and $(C,\sigma|_C)$ is symplectic.
In this exceptional case, this property does not hold.
\end{lem}

\begin{proof}
It is enough to prove that if $x\in\alt(A,\sigma)\cap B$, then $x\in\alt(B,\sigma|_B)$.
Since $x\in\alt(A,\sigma)$, we have $\sigma(x)=-x$.
If $\car F\neq2$, then $x=\frac{1}{2}(x-\sigma(x))\in\alt(B,\sigma|_B)$ and we are done.
So suppose that $\car F=2$ and $(C,\sigma|_C)$ is orthogonal.
Suppose that $\dim_F C=n^2$.
Set $r=\frac{n(n-1)}{2}$ and $s=\frac{n(n+1)}{2}$ (note that $r+s=n^2$).
By \cite[(2.6 (2))]{knus} we have $\dim_F\alt(C,\sigma|_C)=r$ and $\dim_F\sym(C,\sigma|_C)=s$.
Let $\{e_1,\cdots,e_r\}$ be a basis of $\alt(C,\sigma|_C)$ and extend it to a basis $\{e_1,\cdots,e_s\}$ of $\sym(C,\sigma|_C)$.
As $1\notin\alt(C,\sigma|_C)$, we may assume that $e_{r+1}=1$.
Since $e_i\in\alt(C,\sigma|_C)$, $i=1,\cdots,r$, there exists $e_{i+s}\in C$ such that $e_{i+s}-\sigma(e_{i+s})=e_i$.
We claim that $\{e_1,\cdots,e_{n^2}\}$ is a basis of $C$.
Suppose that
\begin{eqnarray}\label{eq1}
\lambda_1e_1+\cdots+\lambda_{n^2}e_{n^2}=0,
\end{eqnarray}
where $\lambda_1,\cdots,\lambda_{n^2}\in F$.
As $\sigma(e_i)=e_i$ for $i=1,\cdots,s$, and $\sigma(e_i)=e_i-e_{i-s}$ for $i=s+1,\cdots,n^2$, we obtain $$\sigma(\lambda_1e_1+\cdots+\lambda_{n^2}e_{n^2})=\lambda_1e_1+\cdots+\lambda_{n^2}e_{n^2}-\lambda_{s+1}e_1-\cdots-\lambda_{n^2}e_{n^2-s}=0.$$
This, together with (\ref{eq1}) implies that $\lambda_{s+1}e_1+\cdots+\lambda_{n^2}e_{n^2-s}=0$.
Note that $n^2-s=r$ and $\{e_1,\cdots,e_r\}$ is a basis of $\alt(C,\sigma|_C)$, so $\lambda_{s+1}=\cdots=\lambda_{n^2}=0$ and we obtain $\lambda_1e_1+\cdots+\lambda_{s}e_{s}=0$.
Now since $\{e_1,\cdots,e_s\}$ is a basis of $\sym(C,\sigma|_C)$, it follows that $\lambda_1=\cdots=\lambda_s=0$.
So the claim is proved.

As $x\in\alt(A,\sigma)$, there exists $y\in A$ such that
$x=y-\sigma(y)$.
Since $(A,\sigma)\simeq (B,\sigma|_B)\otimes(C,\sigma|_C)$,
one can write $y=y_1\otimes e_1+\cdots+y_{n^2}\otimes e_{n^2}$, where $y_1,\cdots,y_{n^2}\in B$.
We have
{\setlength\arraycolsep{1pt}
\begin{eqnarray*}
x&=&y-\sigma(y)\\
&=&y_1\otimes e_1+\cdots+y_{n^2}\otimes e_{n^2}-(\sigma(y_1)\otimes e_1+\cdots+\sigma(y_s)\otimes e_s)\\
&&-(\sigma(y_{s+1})\otimes(e_{s+1}-e_1)+\cdots+\sigma(y_{n^2})\otimes(e_{n^2}-e_{n^2-s}))\\
&=&(y_1-\sigma(y_1)+\sigma(y_{s+1}))\otimes e_1+\cdots+(y_r-\sigma(y_r)+\sigma(y_{s+r}))\otimes e_r\\
&&+(y_{r+1}-\sigma(y_{r+1}))\otimes e_{r+1}+\cdots+(y_{n^2}-\sigma(y_{n^2}))\otimes e_{n^2}.
\end{eqnarray*}}
Since $x\in B$ and $e_{r+1}=1$, we obtain $x=(y_{r+1}-\sigma(y_{r+1}))\otimes1$, i.e., $x\in\alt(B,\sigma|_B)$.

Finally note that in the exceptional case we have $1\in\alt(C,\sigma|_C)$, so if $x\in\sym(B,\sigma|_B)\setminus\alt(B,\sigma|_B)$, then $x\otimes1\in(\alt(A,\sigma)\cap B)\setminus\alt(B,\sigma|_B)$.
\end{proof}

\begin{thm}\label{pfister}
Let $(Q_i,\sigma_i)$, $i=1,\cdots,n$, be a quaternion algebra with involution of the first kind over a field $F$ of characteristic $2$
and let $(A,\sigma)=(Q_1,\sigma_1)\otimes\cdots\otimes(Q_n,\sigma_n)$.
Suppose that $A$ splits.
\begin{itemize}
\item[\text{(a)}]
If $\sigma$ is of symplectic type, then $(A,\sigma)\simeq(M_2(F),\gamma)\otimes\cdots\otimes(M_2(F),\gamma)$.
\item[\text{(b)}]
If $\sigma$ is of orthogonal type, then $(A,\sigma)\simeq(M_2(F),T_{\alpha_1})\otimes\cdots\otimes(M_2(F),T_{\alpha_n})$, where $\alpha_i$, $i=1,\cdots,n$, is a representative of the class $\disc\sigma_i\in F^\times/F^{\times2}$.
\end{itemize}
\end{thm}

\begin{proof}
If $\sigma$ is of symplectic type, the same argument as the case where $\car F\neq2$ works, see \cite[p. 2]{becher}.
So suppose that $\sigma$ is of orthogonal type.
We use induction on $n$.
If $n=1$, the result is trivial.
Suppose that $n\geqslant2$.
By \cite[(2.23)]{knus}, each involution $\sigma_i$ on $Q_i$, $i=1,\cdots,n$, is of orthogonal type.
If $Q_1$ splits then $Q_2\otimes\cdots\otimes Q_n$ splits and the result follows from induction hypothesis.
So suppose that $Q_1$ is a division algebra.
Let $\alpha_i\in F^\times$, $i=1,\cdots,n$, be a representative of the class $\disc\sigma_i\in F^\times/F^{\times2}$.
By (\ref{alt}), there exists $u_1\in\alt(Q_1,\sigma_1)$ such that $\alpha_1=u_1^2\in F^\times$.
Set $u=u_1\otimes1\otimes\cdots\otimes1\in A$.
Then $u^2=\alpha_1$, $\sigma(u)=u$ and $u\in\alt(A,\sigma)$.
Let $C=C_A(u)$ be the centralizer of $u$ in $A$ and let $K=F(u)$.
Then $K/F$ is a quadratic field extension, $Z(C)=K$ and $C\simeq Q_2\otimes\cdots\otimes Q_n\otimes K$ is a central simple $K$-algebra.
Since $\sigma|_K=\id$, we have $(C,\sigma|_C)\simeq_K(Q_2,\sigma_2)\otimes\cdots\otimes(Q_n,\sigma_n)\otimes(K,\id)$.
So
$$(C,\sigma|_C)\simeq_K({Q_2}\otimes K,{\sigma_2}\otimes {\id}_K)\otimes\cdots\otimes({Q_n}\otimes K,{\sigma_n}\otimes {\id}_K).$$
Since $Q_1\otimes K$ splits, $C$ splits.
Also every $({Q_i}\otimes K,{\sigma_i}\otimes {\id}_K)$, $2\leqslant i\leqslant n$, is a quaternion algebra with involution of orthogonal type over $K$ and $\disc({\sigma_i}\otimes {\id}_K)=\alpha_i K^{\times2}$.
So by induction hypothesis we have $(C,\sigma|_C)\simeq(M_2(K),T_{\alpha_2})\otimes\cdots\otimes(M_2(K),T_{\alpha_n})$.
Since $\alpha_i\in F$, by (\ref{b}), there exists a nonalternating symmetric bilinear space $(W_i,b_i)$, $i=2,\cdots,n$, over $K$ with an orthogonal basis $\{u_i,w_i\}$
such that $(M_2(K),T_{\alpha_i})\simeq(\End_K(W_i),\ad_{b_i})$, $b_i(u_i,u_i)=1\in F^\times$ and $b_i(w_i,w_i)=\alpha_i\in F^\times$.
Let $V_i=Fu_i+Fw_i$, $M_i=\End_F(V_i)\subseteq \End_K(W_i)$ and $\rho_i=\ad_{b_i}$, $i=2,\cdots,n$.
Since $b_i(x,x)\in F$ for every $x\in V_i$, the restriction of $b_i$ to $V_i\times V_i$ is a bilinear form on $V_i$.
In particular by (\ref{ad}), we have $\rho_i(M_i)=M_i$.
We also have $\disc\rho_i|_{M_i}=\disc b_i|_{V_i\times V_i}=\alpha_i F^{\times2}$.
So $(M_i,\rho_i|_{M_i})\simeq(M_2(F),T_{\alpha_i})$.
Set $\rho=\rho_2\otimes\cdots\otimes\rho_n$ and $M=M_2\otimes\cdots\otimes M_n\subseteq \End_K(W_2)\otimes\cdots\otimes \End_K(W_n)$.
Then $M$ is a central simple $F$-algebra.
Let $B\subseteq A$ be the image of $M$ under the injection
{\setlength\arraycolsep{2pt}
\begin{eqnarray*}
{\End}_K(W_2)\otimes\cdots\otimes{\End}_K(W_n)\simeq M_2(K)\otimes\cdots\otimes M_2(K)\simeq C\hookrightarrow A,
\end{eqnarray*}}
where the last map is the inclusion $i:C\hookrightarrow A$.
Note that we have $B\subseteq C\subseteq A$ and $B$ is a simple $F$-subalgebra of $A$ with $Z(B)=F$.
Set $Q=C_A(B)$.
Then $Q$ is an $F$-algebra.
Also since $\rho(M)=M$, we have $\sigma(B)=B$ and therefore $\sigma(Q)=Q$.
So
\begin{eqnarray*}
(A,\sigma)&\simeq&(Q,\sigma|_Q)\otimes(B,\sigma|_B)\simeq(Q,\sigma|_Q)\otimes(M,\rho|_M)\\
&\simeq&(Q,\sigma|_Q)\otimes(M_2(F),T_{\alpha_2})\otimes\cdots\otimes(M_2(F),T_{\alpha_n}).
\end{eqnarray*}
Since $A$ splits, $Q$ splits.
So $(Q,\sigma|_Q)\simeq(M_2(F),T_{\alpha})$, where $\alpha\in F^\times$ is a representative of the class $\disc\sigma|_Q\in F^\times/F^{\times2}$.
Finally note that $u\in Z(C)=K$ and $B\subseteq C$, so $u$ commutes with every element of $B$, i.e., $u\in Q$.
As $u\in\alt(A,\sigma)$, we obtain $u\in Q\cap\alt(A,\sigma)$.
It follows from (\ref{discQ}) that $u\in\alt(Q,\sigma|_Q)$.
Also as $u^2=\alpha_1\in F$, we have $\disc\sigma|_Q=\alpha_1F^{\times2}=\disc\sigma_1$ which completes the proof.
\end{proof}

\begin{rem}\label{proof}
The idea of the proof of (\ref{pfister} (b)) does not work if $\car F\neq2$.
In fact if $\car F\neq2$, then for every $u\in\alt(A,\sigma)$ we have $\sigma(u)=-u$.
So the restriction of $\sigma$ to $C=C_A(u)$ is an involution of the second kind and the induction hypothesis cannot be used for $(C,\sigma|_C)$.
\end{rem}

Let $F$ be a field.
An $F$-algebra with involution $(A,\sigma)$ is called {\it isotropic} if there exists $0\neq a\in A$ such that $\sigma(a)a=0$.
Otherwise $(A,\sigma)$ is called {\it anisotropic}.
 An idempotent $e\in A$ is called {\it metabolic} with respect to $\sigma$ if $\sigma(e)e = 0$ and $\dim_FeA = \frac{1}{2}\dim_FA$.
 An algebra with
involution $(A,\sigma)$ is called {\it metabolic} if $A$ contains a metabolic idempotent with respect to $\sigma$.
This notion was first introduced in \cite{berhuy}.
It can be shown that metabolic involutions are adjoint to metabolic hermitian
forms (i.e., hermitian spaces in which there
exists a self orthogonal subspace), see \cite[(4.8)]{dolphin}.
It is known that bilinear Pfister forms, i.e., the forms of the shape
$\langle1,\alpha_1\rangle\otimes\cdots\otimes\langle1,\alpha_n\rangle$, are either anisotropic
or metabolic, see \cite[(6.3)]{elman}.
These, together with (\ref{pfister}) imply that:

\begin{cor}
Let $(A,\sigma)\simeq(Q_1,\sigma_1)\otimes\cdots\otimes(Q_n,\sigma_n)$ be a tensor product of quaternion algebras with involution of the first kind over a field $F$ of characteristic $2$.
If $K$ is a splitting field for $A$, then $(A,\sigma)_K$ is either anisotropic or metabolic.
\end{cor}

Finally, using (\ref{tint}), (\ref{pfister}) can be reformulated as follows:
\begin{cor}
Let $(V,q)$ be a $2n$-dimensional quadratic space over a field $F$ of characteristic $2$ and let $\tau$ be an involution in $O(V,q)$.
Let $m=\dim\fix(V,\tau)-\frac{1}{2}\dim V$.
If $C(V)$ splits, then $(C(V),J_\tau)\simeq(M_2(F),\sigma_1)\otimes\cdots\otimes(M_2(F),$ $\sigma_n)$,
where the involution $\sigma_i$, $i=1,\cdots,n$, on $M_2(F)$ can be chosen as follows:
\begin{itemize}
\item[\text{(a)}]{If $m=0$ and $\tau$ is of reflectional kind, i.e., $\tau=\tau_1\perp\cdots\perp\tau_n$ is an orthogonal sum of $n$ reflections, then $\sigma_i=T_{\alpha_i}$,
 where $\alpha_i\in F$ is a representative of the class $\theta(\tau_i)\in F^{\times}/F^{\times2}$, $i=1,\cdots,n$.}
\item[\text{(b)}]{If $m=0$ and $\tau$ is of interchanging kind, i.e., $\tau$ is an orthogonal sum of interchange isometries, then $\sigma_i=t$, $i=1,\cdots,n$.}
\item[\text{(c)}]{If $m\neq0$, then $\sigma_i=\gamma$, $i=1,\cdots,n$.}
\end{itemize}
\end{cor}

\end{document}